\documentclass[draft]{amsart}

\usepackage{ifthen}

\newtheorem{anyprop}{Anyprop}[section]

\newtheorem{theorem}[anyprop]{Theorem}
\newtheorem{lemma}[anyprop]{Lemma}

\newtheorem{corollary}[anyprop]{Corollary}

\theoremstyle{definition}

\newtheorem{example}[anyprop]{Example}

\newtheorem{remark}[anyprop]{Remark}

\newcommand{\NN}{\mathbb{N}}
\newcommand{\ZZ}{\mathbb{Z}}
\newcommand{\QQ}{\mathbb{Q}}

\newcommand{\FF}{\mathbb{F}}

\newcommand{\PP}{\mathbb{P}}

\newcommand  {\shC}     {\mathcal{C}}

\newcommand  {\shE}     {\mathcal{E}}
\newcommand  {\shF}     {\mathcal{F}}
\newcommand  {\shG}     {\mathcal{G}}

\newcommand  {\shM}     {\mathcal{M}}

\newcommand  {\shL}     {\mathcal{L}}

\newcommand  {\shS}     {\mathcal{S}}

\newcommand  {\shX}     {\mathcal{X}}

\newcommand  {\fom}     {\mathfrak{m}}

\newcommand  {\fop}     {\mathfrak{p}}

\newcommand  {\Char}    {\operatorname{char}}

\newcommand  {\et}      {{\text{\'{e}t}}}

\newcommand  {\length}  {\operatorname{length}}

\renewcommand{\O}       {\mathcal{O}}

\newcommand  {\Proj}    {\operatorname{Proj}}

\newcommand  {\ra}      {\rightarrow}

\newcommand  {\rk}    {\operatorname{rk}}

\newcommand  {\Spec}    {\operatorname{Spec}}

\newcommand  {\Syz}     {\operatorname{Syz}}

\newcommand{\komdots}{ , \ldots , }

\newcommand  {\cloalg}[1]{{ \overline {#1 \!\!}}}

\newcommand{\lto}{\longrightarrow}

\theoremstyle{remark}

\numberwithin{equation}{section}

\usepackage{amscd}
\usepackage{amssymb}
\input xy
\xyoption{all}

\begin{document}
\title[An explicit example of Frobenius periodicity]
{An explicit example of Frobenius periodicity}

\author[Holger Brenner and Almar Kaid]{Holger Brenner and Almar Kaid}
\address{Universit\"at Osnabr\"uck, Fachbereich 6: Mathematik/Informatik, Albrechtstr. 28a,
49069 Osnabr\"uck, Germany}
\email{hbrenner@uni-osnabrueck.de and akaid@uni-osnabrueck.de}





\begin{abstract}
In this paper we show that the restriction of the cotangent bundle $\Omega_{\PP^2}$ of the projective plane to a Fermat curve $C$ of degree $d$ in characteristic $p \equiv -1 \mod 2d$ is, up to tensoration with a certain line bundle, isomorphic to its Frobenius pull-back. This leads to a Frobenius periodicity $F^*(\shE) \cong \shE$ on the Fermat curve of degree $2d$, where $\shE= \Syz(U^2,V^2,W^2)(3)$.
\end{abstract}

\maketitle


\smallskip


\section{Introduction}

Let $C$ be a smooth projective curve defined over a field $K$ of characteristic $p>0$. If $F$ denotes the absolute Frobenius morphism
$F: C \ra C$, then we say that a vector bundle $\shE$ on $C$ admits an \emph{$(s,t)$-Frobenius periodicity} if there are natural numbers
$s$ and $t$, $t>s$, such that ${F^t}^*(\shE) \cong {F^s}^*(\shE)$. Of particular interest are vector bundles which admit a $(0,t)$-Frobenius periodicity, i.e., $F^{t*}(\shE) \cong \shE$. By the classical theorem of H. Lange and U. Stuhler \cite[Satz 1.4]{langestuhler} such a bundle is \'{e}tale trivializable, i.e., there exists an \'{e}tale covering $f:D \ra C$ such that $f^*(\shE) \cong \O_D^r$ where $r = \rk(\shE)$. Hence a vector bundle $\shE$ with a $(0,t)$-Frobenius periodicity comes from a (continuous) representation $\rho: \pi_1^{ \et }(C) \ra GL_r(K)$ of the \'{e}tale fundamental group $\pi_1^{\et}(C)$ of the curve (see [ibid., Proposition 1.2]). We recall that a vector bundle which can be trivialized under an \'etale covering does not necessarily admit a Frobenius periodicity (see \cite[Example 2.10]{brennerkaiddeepfrobeniusdescent} or \cite[Example below Theorem 1.1]{biswasducrohetlangestuhler}).

Quasicoherent modules over a scheme of positive characteristic allowing a Frobenius periodicity appear under several names ($\shF$-finite modules, unit ${\mathcal O}_X[F]$-modules) and from several perspectives ($D$-modules, local cohomology,  Cartier modules, constructible sheaves on the \'{e}tale site, Riemann-Hilbert correspondence) in the literature. Beside \cite{langestuhler} we mention work of Katz \cite[Proposition 4.1.1]{katzmodular}, Lyubeznik \cite{lyubeznikfmodules}, Emerton and Kisin \cite{emertonkisin}, Blickle \cite{blickleminimalgamma} and Blickle and B\"ockle \cite{blickleböcklecartier}.

Despite the importance of vector bundles having a Frobenius periodicity, it is not easy to write down non-trivial explicit examples. For a line bundle the condition becomes ${F^t}^* \shL = \shL^q = \shL$ (with $q=p^t$), hence $\shL$ must be a torsion element in ${\rm Pic} \ C$ of order dividing $q-1$, and so any line bundle of order prime to $p$ gives an example. For higher rank, a necessary condition is that the bundle $\shS$ has degree $0$ and is semistable. By the periodicity it follows that the bundle is in fact strongly semistable, meaning that ${F^t}^*(\shE)$ is semistable for all $t \geq 0$ (this notion goes back to Miyaoka, see \cite[Section 5]{miyaokachern}). On the other hand, if the curve $C$ and the bundle $\shE$ are defined over a finite field and $\shE$ is strongly semistable of degree $0$, then there is necessarily a $(s,t)$-Frobenius periodicity due to the fact that the number of isomorphism classes of semistable vector bundles of fixed rank and degree is finite (\cite[Satz 1.9]{langestuhler}). Nevertheless, it is still hard to detect the periodicity $s$ and $t$. If we have an extension $0\rightarrow \O_C \rightarrow \shS \rightarrow \O \rightarrow 0$ given by $c \in H^1(C,\O_C)$, then its Frobenius pull-back is given by the class $F^*(c)$, and one can get (semistable, but not stable) examples by looking at the Frobenius action on $H^1(C,\O_C)$.

In this note we provide a down to earth example of a stable rank-$2$ vector bundle $\shE$ on a suitable Fermat curve admitting a $(0,1)$-Frobenius periodicity. Moreover, this periodicity only depends on a congruence condition of the characteristic of the base field, not on its algebraic structure. Our main tools will be results of P. Monsky on the Hilbert-Kunz multiplicity of Fermat hypersurface rings and the geometric approach to Hilbert-Kunz theory developed independently by the first author in \cite{brennerhilbertkunz} and V. Trivedi in \cite{trivedihilbertkunz}.

The results of this article are contained in Chapter $4$ of the PhD-thesis \cite{kaiddissertation} of the second author.
Related results on the free resolution of Frobenius powers on a Fermat ring can be found in the preprint \cite{kustinrahmativraciu}.
We thank Manuel Blickle, Aldo Conca, Neil Epstein, Andrew Kustin and Axel St\"abler for useful discussions. We also thank the referee for useful remarks.

\section{A lemma on global sections}

To begin with, we recall the notions of a \emph{syzygy bundle}. Let $K$ be a field and let $R$ be a normal standard-graded $K$-domain of dimension $d \geq 2$. Then homogeneous $R_+$-primary elements $f_1 \komdots f_n$ (i.e., $\sqrt{(f_1 \komdots f_n)}=R_+$)
of degrees $d_1 \komdots d_n$ define a short exact (presenting) sequence 
$$0 \lto \Syz(f_1 \komdots f_n) \lto \bigoplus_{i=1}^n \O_X(-d_i) \stackrel{f_1 \komdots f_n}\lto \O_X \lto 0$$
on the projective scheme $X= \Proj R$. The kernel $\Syz(f_1 \komdots f_n)$ is locally free and is called the \emph{syzygy bundle} for the elements $f_1 \komdots f_n$.

In this article we only deal with restrictions of syzygy bundles of the form $\Syz(X^a,Y^a,Z^a)$, $a \in \NN \setminus \{0\}$, on $\PP^2=\Proj K[X,Y,Z]$ to a plane curve $C$. Our main interest will be the case $a=1$ which corresponds via the Euler sequence to the cotangent bundle $\Omega_{\PP^2}|_C$ on the projective plane. Since there will be no confusion in the sequel we also denote the restricted bundle on the curve by $\Syz(X^a,Y^a,Z^a)$.

Let $K$ be a field and consider a smooth plane curve of the form $$V_+(Z^d -P(X,Y)) \subset \PP^2= \Proj K[X,Y,Z],$$ where $P(X,Y) \in K[X,Y]$ denotes a homogeneous polynomial of degree $d$. In this situation we can compute global sections of a rank-$2$ syzygy bundle of the form $\Syz(X^{a_1},Y^{a_2},Z^{a_3})$ by the following lemma which is a slight improvement over \cite[Lemma 1]{brennermiyaoka}. It relates the sheaves $\Syz(X^{a_1},Y^{a_2},Z^{a_3})$ with the sheaves $\Syz(X^{a_1},Y^{a_2},P(X,Y)^k)$ for suitable $k$ which come from $\PP^1$ via the Noetherian normalization $C=V_+(Z^d-P(X,Y)) \ra \PP^1=\Proj K[X,Y]$. We will use this result several times in the proof of our main theorem in the next section.

\begin{lemma}
\label{splittinglemma}
Let $K$ be a field and let $P(X,Y) \in K[X,Y]$ be a homogeneous
polynomial of degree $d$. Suppose the plane curve
$$C:=\Proj(K[X,Y,Z]/(Z^d-P(X,Y)))$$
is smooth. Further, fix
$a_1,a_2,a_3 \in \NN_+$ and write $a_3=dk+t$ with $0 \leq t < d$. Then we
have for every $m \in \ZZ$ a surjective sheaf morphism
\begin{eqnarray*}
\varphi_m:\shS_k(m-t) \oplus \shS_{k+1}(m) &\lto&
\Syz(X^{a_1},Y^{a_2},Z^{a_3})(m)\\
(f_1,f_2,f_3),(g_1,g_2,g_3) &\longmapsto&
(Z^tf_1+g_1,Z^tf_2+g_2,f_3+Z^{d-t}g_3) \, ,
\end{eqnarray*}
where
$\shS_i:=\Syz(X^{a_1},Y^{a_2},P(X,Y)^i)$ for $i \geq 0$. Moreover,
the corresponding map on global sections
$$\Gamma(C,\shS_k(m-t)) \oplus \Gamma(C,\shS_{k+1}(m)) \lto
\Gamma(C,\Syz(X^{a_1},Y^{a_2},Z^{a_3})(m))$$ is surjective for
every $m \in \ZZ$.
\end{lemma}

\begin{proof}
We consider the sheaf morphism
$$ \xymatrix {&\O_C(m-t-a_1) \oplus
\O_C(m-t-a_2) \oplus \O_C(m-t-dk) \ar@{}@<-1.5ex>[d]^{\bigoplus} \\
              &\O_C(m-a_1) \oplus \O_C(m-a_2)\oplus \O_C(m-dk-d)
              \ar[d]^{\widetilde{\varphi_m}}
              &\\
              &\O_C(m-a_1) \oplus
\O_C(m-a_2)\oplus \O_C(m-a_3) &}
              $$
which maps $(s_1,s_2,s_3,s_4,s_5,s_6) \mapsto
(Z^ts_1+s_4,Z^ts_2+s_5,s_3+Z^{d-t}s_6)$. Clearly,
$\widetilde{\varphi_m}$ maps $\shS_k(m-t)\oplus\shS_{k+1}(m)$ into
$\Syz(X^{a_1},Y^{a_2},Z^{a_3})(m)$. Hence, the map $\varphi_m$ is
obtained from $\widetilde{\varphi_m}$ via restriction to
$\shS_k(m-t)\oplus\shS_{k+1}(m)$ and is therefore a morphism of
sheaves. It is enough to prove that $\varphi_m$ is surjective on global sections for all $m$. Let $s:=(F,G,H) \in
\Gamma(C,\Syz(X^{a_1},Y^{a_2},Z^{a_3})(m))$ be a non-trivial global section, i.e.,
$FX^{a_1}+GY^{a_2}+HZ^{a_3}=0$ and $\deg(F) + a_1 = \deg(G) + a_2
=\deg(H) + a_3 =m$. We write
\begin{eqnarray*}
F&=&F_0 + F_1 Z + F_2 Z^2 + \ldots + F_{d-1}
Z^{d-1}\\
G&=&G_0 + G_1 Z + G_2 Z^2 + \ldots + G_{d-1}
Z^{d-1}\\
H&=&H_0 + H_1 Z + H_2 Z^2 + \ldots + H_{d-1} Z^{d-1}
\end{eqnarray*}
with $F_i,G_i,H_i \in K[X,Y]$ for $i=0 \komdots d-1$. We have
$Z^{a_3}=Z^{dk+t}=(Z^d)^kZ^t=P(X,Y)^k Z^t$. Since $s$ is a syzygy we obtain (by considering the $K[X,Y]$-component corresponding to $Z^{i}$) a system of equations
$$F_iZ^i X^{a_1}+G_iZ^iY^{a_2}+H_{j(i)}Z^{j(i)}Z^{a_3}=0,$$
where $j(i) \equiv i-t \!\! \mod d$. Thus $s=(F,G,H)$ is the sum of the
syzygies $$ s_i:=(F_iZ^i,G_iZ^i,H_{j(i)}Z^{j(i)}) \in
\Gamma(C,\Syz(X^{a_1},Y^{a_2},Z^{a_3})(m)).$$
We show that each of these summands does either come from $\Gamma(C,\shS_{k+1}(m))$ or from $\Gamma(C,\shS_k(m-t))$. We fix one equation
$$F_{i_0}Z^{i_0} X^{a_1}+G_{i_0}Z^{i_0}Y^{a_2}+H_{j(i_0)}Z^{j(i_0)}Z^{a_3}=0$$
with $j(i_0) \equiv i_0-t \!\! \mod d$. First, we treat the case where $i_0 <t$, hence $j(i_0)=i_0 -t+d$. Factoring out $Z^{i_0}$ and replacing $Z^{a_3}$ by $P(X,Y)^k Z^t$ yields
\begin{eqnarray*}
0 &=& Z^{i_0}(F_{i_0}X^{a_1}+G_{i_0}Y^{a_2}+H_{j(i_0)}Z^{d-t}P(X,Y)^kZ^t)\\
&=& Z^{i_0}(F_{i_0}X^{a_1}+G_{i_0}Y^{a_2}+H_{j(i_0)}
P(X,Y)^{k+1}).
\end{eqnarray*}
Hence $g_{i_0}:=(Z^{i_0}F_{i_0},Z^{i_0}G_{i_0},Z^{i_0}H_{j(i_0)}) \in \Gamma(C,\shS_{k+1}(m))$ and $\varphi_m(g_{i_0})=s_{i_0}$.
Next, we consider the case $i_0 \geq t$, hence $j(i_0)=i_0-t$. We factor out
$Z^t$ and replace $Z^{a_3}$. This gives
\begin{eqnarray*}
0 &=& F_{i_0} Z^{j(i_0)+t}X^{a_1}+ G_{i_0}Z^{j(i_0)+t}Y^{a_2}+H_{j(i_0)}Z^{j(i_0)} P(X,Y)^k Z^t\\
&=& Z^t(F_{i_0}Z^{j(i_0)}X^{a_1}+G_{i_0}Z^{j(i_0)}Y^{a_2}+H_{j(i_0)}Z^{j(i_0)}P(X,Y)^k).
\end{eqnarray*}
Hence we have $h_{i_0}:=(F_{i_0}Z^{j(i_0)},G_{i_0}Z^{j(i_0)},H_{j(i_0)}Z^{j(i_0)}) \in \Gamma(C,\shS_k(m-t))$ and $\varphi_m(h_{i_0})=s_{i_0}$.
\end{proof}

\begin{remark}
\label{summandinjectivity}
It is easy to see that the morphisms $\varphi_m$, $m \in \ZZ$, are injective on both summands, i.e., the induced mappings
$$\shS_k(m-t) \lto \Syz(X^{a_1},Y^{a_2},Z^{a_3})(m), ~(f_1,f_2,f_3) \longmapsto (Z^tf_1,Z^tf_2,f_3)$$
and
$$\shS_{k+1}(m) \lto \Syz(X^{a_1},Y^{a_2},Z^{a_3})(m), ~(g_1,g_2,g_3) \longmapsto (g_1,g_2,Z^{d-t}g_3)$$
are both injective.
\end{remark}

\begin{remark}
\label{splittingonp1remark}
The sheaves $\shS_k$ and $\shS_{k+1}$ are the pull-backs $$\pi^*(\Syz_{\PP^1}(X^{a_1},Y^{a_2},P(X,Y)^k)) \mbox{ and } \pi^*(\Syz_{\PP^1}(X^{a_1},Y^{a_2},P(X,Y)^{k+1}))$$ respectively under the Noetherian normalization $\pi: C \ra \PP^1= \Proj K[X,Y]$. In particular, $\shS_k$ and $\shS_{k+1}$ split as a direct sum of line bundles. If $t=0$ we have $\Syz_C(X^{a_1},Y^{a_2},Z^{a_3}) \cong \Syz_C(X^{a_1},Y^{a_2},P(X,Y)^k)$ and the bundle is therefore already defined
on $\PP^1$.
\end{remark}

\section{Frobenius periodicity up to a twist}

Let $C$ be a smooth projective curve defined over a field of positive characteristic. It is a well-known fact that the pull-back of
a semistable vector bundle under the (absolute) Frobenius morphism is in general not semistable anymore; see for instance the example of Serre in \cite[Example 3.2]{hartshorneamplecurve}. Using syzygy bundles on Fermat curves one can produce fairly easy examples of this phenomenon.

\begin{example}
\label{destabilizingexample}
Let $C:=\Proj(\cloalg{\FF_3}\,\,[X,Y,Z]/(X^4+Y^4-Z^4)$ be the Fermat quartic in characteristic $3$. The cotangent bundle $\Omega_{\PP^2}$ is stable on the projective plane (see for instance \cite[Corollary 6.4]{brennerlookingstable})
and so is the restriction $\Omega_{\PP^2}|_C = \Syz(X,Y,Z)$ by Langer's restriction theorem \cite[Theorem 2.19]{langersurvey}.
Its Frobenius pull-back is the syzygy bundle $\Syz(X^3,Y^3,Z^3)$. The curve equation yields the relation $X \cdot X^3 + Y \cdot Y^3 - Z \cdot Z^3 = 0$
and thus we obtain a non-trivial global section of $(F^*(\Omega_{\PP^2}|_C))(4)$. But the degree of this bundle equals $-4$ and therefore $F^*(\Omega_{\PP^2}|_C)$ is not semistable.
\end{example}

Before we state our main theorem, we prove the following Lemma separately.

\begin{lemma}
\label{stronglysemistablerestriction}
Let $d \geq 2$ be an integer and let $K$ be a field of characteristic $p \equiv - 1 \!\! \mod 2d$. Then
$\Omega_{\PP^2}|_C$ is strongly semistable on the Fermat curve $C := \Proj(K[X,Y,Z]/(X^d+Y^d-Z^d))$.
\end{lemma}

\begin{proof}
We use Hilbert-Kunz theory and its geometric interpretation developed in \cite{brennerhilbertkunz} and \cite{trivedihilbertkunz}.
The Hilbert-Kunz multiplicity $e_{HK}(R)$ of the homogeneous coordinate ring
$R:=K[X,Y,Z]/(X^d+Y^d-Z^d)$ of the Fermat curve equals $\frac{3d}{4}$ in characteristic $p \equiv - 1 \!\! \mod 2d$ by Monsky's result \cite[Theorem 2.3]{monskytrinomial}. By \cite[Corollary 4.6]{brennerhilbertkunz} this is equivalent to the strong semistability of $\Omega_{\PP^2}|_C$ in these characteristics.
\end{proof}

\begin{remark}
Note that for $d=1$ we have $C \cong \PP^1$ and $\Omega_{\PP^2}|_C \cong \O_C(-2) \oplus \O_C(-1)$, i.e., $\Omega_{\PP^2}|_C$ is not even semistable.
For a general characterization of strong semistability of $\Omega_{\PP^2}|_C$ on the Fermat curve of degree $d$ depending on the characteristic of the
base field see \cite[Chapter 4]{kaiddissertation}. The restriction of $\Omega_{\PP^2} $ to every smooth projective curve of degree $\geq 7$
is stable by Langer's restriction theorem \cite[Theorem 2.19]{langersurvey}.
\end{remark}

\begin{theorem}
\label{cotangentfermatrepetition}
Let $d \geq 2$ be an integer and let $K$ be a field of characteristic $p \equiv - 1 \!\! \mod 2d$. Then $\shE:=\Omega_{\PP^2}|_C$ is strongly semistable on the Fermat curve $C := \Proj(K[X,Y,Z]/(X^d+Y^d-Z^d))$ and $$F^*(\shE) \cong \shE(-\frac{3(p-1)}{2}).$$
\end{theorem}

\begin{proof}
The strong semistability of $\shE$ in characteristics $p \equiv -1 \!\! \mod 2d$ has already been proved in Lemma \ref{stronglysemistablerestriction}. So we have to show that $F^*(\shE) \cong \Syz(X^p,Y^p,Z^p) \cong \shE(-\frac{3(p-1)}{2})$. Since the proof is quite long, we divide it into several steps. Note that, since semistability is preserved under base change, we may assume without loss of generality that $K$ is algebraically closed.

\smallskip

\noindent \underline{Step 1.} We write $p = dk + (d-1)$ with $k$ odd. Accordingly, we set $t=d-1$. Further, we follow the notation of Lemma \ref{splittinglemma} and define the bundles
$$\shS_k:=\Syz(X^p,Y^p,(X^d+Y^d)^k),~~\shS_{k+1}:=\Syz(X^p,Y^p,(X^d+Y^d)^{k+1}).$$
We show that the surjective morphism
$$\varphi_{\frac{3p+1}{2}}: \shS_k(\frac{3p+1}{2}-t) \oplus \shS_{k+1}(\frac{3p+1}{2}) \lto \Syz(X^p,Y^p,Z^p)(\frac{3p+1}{2})$$ defined in Lemma \ref{splittinglemma} can be identified with
$$ (\O_C(-d+2) \oplus \O_C) \oplus \O_C^2 \lto \Syz(X^p,Y^p,Z^p)(\frac{3p+1}{2}).$$

\smallskip

We consider the vector bundle $\Syz(U^{k+1},V^{k+1},(U+V)^k)(\frac{3k+1}{2})$ on the projective line $\PP^1 = \Proj K[U,V]$. Since the degree of this bundle is $-1$, it has to have a non-trivial global section. Substituting $U=X^d$ and $V=Y^d$ yields a non-trivial syzygy
$$F X^{dk+d} + G Y^{dk+d} + H (X^d+Y^d)^k = (FX) X^p + (GY) Y^p + H (X^d+Y^d)^k = 0$$
of total degree $\frac{3dk+d}{2}$. That is, we have a non-trivial global section of $\shS_k(\frac{3dk+d}{2})$ on the curve $C$. We have $\Gamma(C,\shS_k(\frac{3dk+d}{2}-1)) = 0$ because otherwise the twisted semistable Frobenius pull-back $\Syz(X^p,Y^p,Z^p)(\frac{3dk+d}{2}+d-2)$ of degree $-d$ would have a non-trivial global section too (see Remark \ref{summandinjectivity}) which is impossible by semistability. Since $\deg(\shS_k(\frac{3dk+d}{2}))= (-d+2)d$, we obtain the splitting (rewrite $\frac{3dk+d}{2}=\frac{3p+1}{2}-(d-1)=\frac{3p+1}{2}-t$)
$$\shS_k(\frac{3p+1}{2}-t) \cong \O_C(-d+2) \oplus \O_C.$$
The other summand $\shS_{k+1}(\frac{3dk+d}{2} + d-1)$ has degree $0$. It follows once again from Lemma \ref{splittinglemma} and the semistability of $\Syz(X^p,Y^p,Z^p)$
that $$\Gamma(C, \shS_{k+1}(\frac{3dk+d}{2} + d-2)) = 0,$$ i.e., $\shS_{k+1}(\frac{3dk+d}{2} + d-1)$ splits as (rewrite $\frac{3dk+d}{2} + d-1 = \frac{3p+1}{2}$)
$$\shS_{k+1}(\frac{3p+1}{2}) \cong \O_C^2.$$

\smallskip

\noindent \underline{Step 2.} Let $(FX,GY,H)$ be the non-trivial global section of $\shS_k(\frac{3p+1}{2}-t)$ constructed above (corresponding to the component $\O_C$). We show that $H(P) \neq 0$ for every point $P=(x,y,z) \in C$ satisfying $z^d=x^d+y^d=0$.

\smallskip

The last component $H$ of the section $(FX,GY,H)$ is a homogeneous polynomial in $X^d$ and $Y^d$ (it stems by construction from a syzygy on $\PP^1$ in $U$ and $V$). Let $P=(x,y,z) \in C$ be a point on the curve such that $z^d=x^d+y^d=0$. Then $x^d=-y^d$ which implies $x= \zeta y$ where $\zeta$ is a $d$th root of $-1$. In particular, $P=(\zeta y,y,0)$. Since $K$ is algebraically closed, $\Char(K) \neq 2$ and $p \equiv -1 \!\! \mod 2d$, the group $\mu_{2d}(K)$ of $(2d)$th roots of unity in $K$ has order $2d$. Hence, we have $$X^d+Y^d = \prod_\zeta (X- \zeta Y),$$ where $\zeta \in \mu_{2d}(K)$ runs through the elements with the property $\zeta^d=-1$ (there are exactly $d$ such roots). Now assume $H(P) = 0$. Then $H(P^\prime)=0$ for all points $P^\prime$ of the form $P^\prime = (\zeta y,y,0)$. So $X^d+Y^d$ has to divide $H$, i.e., $H= \tilde{H} (X^d+Y^d)$ with a homogeneous polynomial $\tilde{H} \in K[X,Y]$. So we have a relation
$$(FX) X^p + (GY) Y^p + \tilde{H} (X^d+Y^d)^{k+1} = 0$$
of total degree $\frac{3p+1}{2}-t$. That is, we have a non-trivial section of the bundle $\shS_{k+1}(\frac{3p+1}{2}-t)$. This section maps by Lemma \ref{splittinglemma} and Remark \ref{summandinjectivity} to a non-trivial global section of
$\Syz(X^p,Y^p,Z^p)(\frac{3p+1}{2}-t)$. But
$$\deg(\Syz(X^p,Y^p,Z^p)(\frac{3p+1}{2}-t))=(3p+1-2t-3p)d=(1-2t)d<0.$$ Hence, the section contradicts the semistability of $\Syz(X^p,Y^p,Z^p)$.

\smallskip

\noindent \underline{Step 3.} We show that in the surjective sheaf homomorphism
$$ \varphi_{\frac{3p+1}{2}}: (\O_C(-d+2) \oplus \O_C) \oplus \O_C^2 \lto \Syz(X^p,Y^p,Z^p)(\frac{3p+1}{2})$$
the summand $\O_C(-d+2)$ is not necessary, i.e.,
$$ \varphi_{\frac{3p+1}{2}}: \O_C^3= \O_C \oplus \O_C^2 \lto \Syz(X^p,Y^p,Z^p)(\frac{3p+1}{2})$$
is also surjective.

\smallskip

Set $m:=\frac{3p+1}{2}$. By the Nakayama lemma, we can check surjectivity pointwise over the residue field $K$ at every point $P=(x,y,z) \in C$. For this we have to find two linearly independent vectors in the image. First we treat the case $z \neq 0$. We show that we even have a surjective map
$$ \shS_{k+1}(m) = \O_C^2 \lto \Syz(X^p,Y^p,Z^p)(m).$$
We take basic sections $$f=(f_1,f_2,f_3), g=(g_1,g_2,g_3) \in \Gamma(C, \shS_{k+1}(m)) \cong \Gamma(C,\O_C^2).$$
Their images are $\tilde{f}=(f_1,f_2,Zf_3)$ and $\tilde{g}=(g_1,g_2,Zg_3)$. Assume there is a relation $\tilde{f}(P) + \lambda \tilde{g}(P) = 0$ with $\lambda \in K^\times$. Looking at each component this gives the equations
\begin{eqnarray*}
f_1(P) + \lambda g_1(P) &=& 0,\\
f_2(P) + \lambda g_2(P)&=& 0,\\
(f_3(P) + \lambda g_3(P)) z &=& 0.
\end{eqnarray*}
But since $z \neq 0$, the latter equation would mean $f_3(P) + \lambda g_3(P)=0$ and therefore we would obtain a relation $f(P) + \lambda g(P) =0$ which contradicts the assumption.

Now we deal with the case $z = 0$, i.e., $P=(x,y,0)$. Let
$$f=(FX,GY,H) \in \Gamma(C, \shS_k(m-t)) \cong \Gamma(C,\O_C(-d+2) \oplus \O_C)$$
be the section (corresponding to $\O_C$ which we have found in step 1. The image of $f$ in the bundle $\Syz(X^p,Y^p,Z^p)(m)$ is the section $(Z^tFX,Z^t GY,H)$. Evaluated at $P$ we obtain the vector $v=(0,0,H(P))$. Since $0 = z^d= x^d + y^d$ we have $H(P) \neq 0$ by step 2. Now we take a section $0 \neq g=(g_1,g_2,g_3) \in
\Gamma(C, \shS_{k+1}(m)) \cong \Gamma(C,\O_C^2)$. The image of $g$ equals $(g_1,g_2,Zg_3)$. Evaluation at $P$ gives  the vector $w=(g_1(P),g_2(P),0)$, where either $g_1(P)$ or $g_2(P)$ is not $0$ (otherwise $g_3(P)$ would be $0$ as well). Hence we have found a vector $w$ such that $v,w$ are linearly independent over $K$.

\smallskip

\noindent \underline{Step 4.} So far we have shown that we have a surjective morphism
$$ \O_C^3 \lto \Syz(X^p,Y^p,Z^p)(\frac{3p+1}{2}) \lto 0.$$
Since $\det(\Syz(X^p,Y^p,Z^p)(\frac{3p+1}{2}))\cong \O_C(1)$ we have a short exact sequence
$$0 \lto \O_C(-1) \lto \O_C^3 \lto \Syz(X^p,Y^p,Z^p)(\frac{3p+1}{2}) \lto 0.$$
Dualizing and tensoring with $\O_C(-1)$ gives
$$0 \lto (\Syz(X^p,Y^p,Z^p)(\frac{3p+1}{2}))^\vee(-1) \lto \O_C^3(-1) \lto \O_C \lto 0,$$
where the map $\O_C^3(-1) \ra \O_C$ is given by some linear forms $L_1,L_2,L_3$ in the homogeneous coordinate ring $R = K[X,Y,Z]/(X^d+Y^d-Z^d)$. In particular, we have $(\Syz(X^p,Y^p,Z^p)(\frac{3p+1}{2}))^\vee(-1) \cong \Syz(L_1,L_2,L_3)$. We show that $\{L_1,L_2,L_3\}$ and $\{X,Y,Z\}$ generate the same ideal in $R$. Assume to the contrary that $L_1,L_2,L_3$ are linearly dependent. Such an equation yields a non-trivial section of $\Syz(L_1,L_2,L_3)(1)$. This bundle has degree $\deg(\Syz(L_1,L_2,L_3)(1))=(2-3)d=-d <0$. But since $\Syz(X^p,Y^p,Z^p)$ is semistable, so is $(\Syz(X^p,Y^p,Z^p)(\frac{3p+1}{2}))^\vee$ and thus also $\Syz(L_1,L_2,L_3)$. So the section contradicts the semistability.

\smallskip

\noindent \underline{Step 5.} We have already proved that
$$\shE \cong \Syz(X,Y,Z) \cong (\Syz(X^p,Y^p,Z^p)(\frac{3p+1}{2}))^\vee(-1).$$ Since $\Syz(X^p,Y^p,Z^p)$ is a bundle of
rank $2$, we have $$\Syz(X^p,Y^p,Z^p) \cong \Syz(X^p,Y^p,Z^p)^\vee \otimes \O_C(-3p).$$
So finally we obtain
\begin{eqnarray*}
\shE &\cong& \Syz(X,Y,Z)\\
&\cong& (\Syz(X^p,Y^p,Z^p)(\frac{3p+1}{2}))^\vee(-1)\\
&\cong& \Syz(X^p,Y^p,Z^p)^\vee \otimes \O_C(-\frac{3p+1}{2}) \otimes \O_C(-1)\\
&\cong& \Syz(X^p,Y^p,Z^p) \otimes \O_C(3p) \otimes \O_C(-\frac{3p+1}{2}) \otimes \O_C(-1)\\
&\cong& \Syz(X^p,Y^p,Z^p)(\frac{3(p-1)}{2}),
\end{eqnarray*}
and consequently $F^*(\shE) \cong \Syz(X^p,Y^p,Z^p) \cong \shE(-\frac{3(p-1)}{2})$ which finishes the proof.
\end{proof}

\begin{remark}
We also comment on the case $p \equiv 1 \!\! \mod 2d$. Then we write $p = dk +1$ with $k$ even and set $t=1$. The syzygy bundle
$\Syz(U^k,V^k,(U+V)^{k+1})(\frac{3k}{2})$ on $\PP^1=\Proj K[U,V]$ has degree $-1$ and therefore has to have a non-trivial global section. Substituting $U=X^d$ and $Y^d$ and multiplying with $XY$ gives then a syzygy
$$(FY) X^p +(GX) Y^p + (HXY)(X^d +Y^d)^{k+1}=0,$$
i.e., a global section of $\shS_{k+1}(\frac{3dk}{2}+2)$ on the Fermat curve. As in the proof of Theorem \ref{cotangentfermatrepetition} we obtain the splittings (rewrite $\frac{3dk}{2}+2 = \frac{3p+1}{2}$)
$$ \shS_{k+1}(\frac{3p+1}{2}) \cong \O_C(-d+2) \oplus \O_C \mbox{ and } \shS_k(\frac{3p+1}{2}-t) \cong \O_C^2.$$
Unfortunately, we do not know how to prove an analog of step 2 in these characteristics, i.e., to show that $H(P) \neq 0$ for every point $P=(x,y,z) \in C$ with $z^d= x^d+y^d =0$. Here the reasoning of the proof (step 2) of Theorem \ref{cotangentfermatrepetition} would lead to a section of $\Syz(X^p,Y^p,(X^d+Y^d)^{k+2})$ which is not helpful to get a contradiction.
\end{remark}

\begin{remark}
\label{periodicityinchartwo}
We cannot expect that Theorem \ref{cotangentfermatrepetition} holds in every characteristic $p$ where $\Omega_{\PP^2}|_C$
is strongly semistable. For example, consider in characteristic $2$ the Fermat cubic $C=\Proj(K[X,Y,Z]/(X^3+Y^3-Z^3))$, which is an elliptic curve. It is a well-known fact that semistable vector bundles on elliptic curves are strongly semistable (see for instance \cite[appendix]{tusemistable}). Hence
$\Omega_{\PP^2}|_C \cong \Syz(X,Y,Z)$ is strongly semistable by \cite[Proposition 6.2]{brennercomputationtight}. The pull-back $F^*(\Omega_{\PP^2}|_C) \cong \Syz(X^2,Y^2,Z^2)$ has for the first time non-trivial global sections in total degree $3$, namely the (only) syzygy $(X,Y,-Z)$ which comes from the equation of the curve. This section gives rise to the short exact sequence
$$0 \lto \O_C \lto \Syz(X^2,Y^2,Z^2)(3) \lto \O_C \lto 0,$$
i.e., $\Syz(X^2,Y^2,Z^2)(3)$ is the bundle $F_2$ in Atiyah's classification \cite{atiyahelliptic}. Since the Hasse invariant of $C$ is $0$, we have $F^*(F_2) \cong \O_C^2$ and therefore $F^*(\Omega_{\PP^2}|_C) \not \cong \Omega_{\PP^2}|_C(-\frac{3(p-1)}{2})$. We have
${F^2}^*(\Omega_{\PP_2}|_C) \cong \O_C(-6) \oplus \O_C(-6)$ and we obtain (up to a twist) the periodicity ${F^3}^*(\Omega_{\PP^2}|_C) \cong ({F^2}^*(\Omega_{\PP^2}|_C))(-6)$.
\end{remark}

\section{A computation of the Hilbert-Kunz function}

We recall that the \emph{Hilbert-Kunz function} of a standard graded ring $R$ of characteristic $p>0$ with graded maximal ideal $\fom$ is the function
$$e \longmapsto \varphi_R(e):=\length (R/\fom^{[p^e]}),$$
where $\fom^{[p^e]}$ denotes the extended ideal under the $e$-th iteration of the Frobenius endomorphism on $R$; see for instance \cite{monskyhilbertkunz} for this rather complicated function and its properties. As a consequence of Theorem \ref{cotangentfermatrepetition} we obtain the complete Hilbert-Kunz function of the Fermat ring $R=K[X,Y,Z]/(X^d+Y^d-Z^d)$ in characteristics $p \equiv -1 \!\!\mod 2d$. The following result is implicitly contained in \cite[Lemma 5.6]{hanmonsky} of P. Monsky and C. Han.

\begin{corollary}
\label{hkfunctionfermatperiodicity}
Let $d \geq 2$ be a positive integer and let $K$ be a field of characteristic $p \equiv -1 \mod 2d$. Then the Hilbert-Kunz function of the Fermat ring $R=K[X,Y,Z]/(X^d+Y^d-Z^d)$ is
$$\varphi_R(e) = \frac{3d}{4}p^{2e} + 1-\frac{3d}{4}.$$
\end{corollary}

\begin{proof}
Since the length of $R/\fom^{[p^e]}$, $\fom=(X,Y,Z)$, does not change if one enlarges the base field, we may assume that $K$ is algebraically closed. Hence, $\varphi_R(e)=\sum_{m=0}^\infty \dim_K(R/\fom^{[p^e]})_m$ (this sum is in fact finite since the algebras $R/\fom^{[p^e]}$ have finite length).
It follows from the presenting sequence of $\Syz(X,Y,Z)$ on the Fermat curve $C= \Proj R$ that (setting $q:=p^e$)
\begin{eqnarray*}
(*) \qquad \dim_K(R/\fom^{[q]})_m &=& h^0(C,\O_C(m))-3 h^0(C,\O_C(m-q))\\&\mbox{}& + h^0(C,\Syz(X^q,Y^q,Z^q)(m)).
\end{eqnarray*}
By Theorem \ref{cotangentfermatrepetition} we have $\Syz(X^p,Y^p,Z^p) \cong \Syz(X,Y,Z)(-\frac{3(p-1)}{2})$ and consequently
$\Syz(X^q,Y^q,Z^q)\cong \Syz(X,Y,Z)(-\frac{3(q-1)}{2})$ for all $q=p^e$, $e \geq 1$. The global evaluation of the presenting sequence of $\shE(k):=\Syz(X,Y,Z)(k)$ gives the exact sequence
$$0 \lto \Gamma(C,\shE(k)) \lto \Gamma(C,\O_C(k-1)^3) \lto \Gamma(C,\O_C(k)) \lto K \lto 0$$
for $k=0$ and the short exact sequence
$$0 \lto \Gamma(C,\shE(k)) \lto \Gamma(C,\O_C(k-1)^3) \lto \Gamma(C,\O_C(k)) \lto 0$$
for $k \geq 1$. Hence we obtain
\[h^0(C,\shE(k))=\begin{cases} 3h^0(C,\O_C(k-1))-h^0(C,\O_C(k))+1 \mbox{ if } k=0,\\
3h^0(C,\O_C(k-1))-h^0(C,\O_C(k)) \mbox{ if } k \neq 0.
\end{cases} \]
For $k \geq d-2$ we have by Riemann-Roch $h^0(C,\O_C(k))=dk-g+1$, where $g$ is the genus of the curve. Since $p \equiv -1 \!\! \mod 2d$, this holds in particular for $k \geq \frac{p+1}{2}$. So the geometric formula for the Hilbert-Kunz function $(*)$ gives (in order to obtain an easier calculation we sum up to $2q$):
\begin{eqnarray*}
\varphi_R(e) &=& \sum_{m=0}^{2q}h^0(C,\O_C(m))-3 \sum_{m=0}^{2q} h^0(C,\O_C(m-q))\\
&\mbox{}& + 3\sum_{m=0}^{2q} h^0(C,\O_C(m-\frac{3(q-1)}{2}-1))\\
&\mbox{}& - \sum_{m=0}^{2q} h^0(C,\O_C(m-\frac{3(q-1)}{2})) + 1\\
&=& \sum_{m=0}^{2q}h^0(C,\O_C(m))-3 \sum_{m=0}^{q} h^0(C,\O_C(m))\\
&\mbox{}& + 3\sum_{m=0}^{\frac{q+1}{2}} h^0(C,\O_C(m))-\sum_{m=0}^{\frac{q+3}{2}} h^0(C,\O_C(m)) + 1\\
&=& \sum_{m=\frac{q+5}{2}}^{2q} h^0(C,\O_C(m))-3 \sum_{m=\frac{q+3}{2}}^{q} h^0(C,\O_C(m)) + 1\\
&=& \sum_{m=\frac{q+5}{2}}^{2q} (dm-g+1)-3 \sum_{m=\frac{q+3}{2}}^{q} (dm-g+1) + 1 \\
&=&  d\left(q(2q+1)-\frac{(q+3)(q+5)}{8}\right)-\frac{3g(q-1)}{2} + \frac{3(q-1)}{2}\\
&\mbox{}&-3d\left(\frac{q(q+1)}{2}-\frac{(q+1)(q+3)}{8}\right)+\frac{3g(q-1)}{2}-\frac{3(q-1)}{2} +1\\
&=&\frac{3d}{4}q^2 + 1 - \frac{3d}{4}.
\end{eqnarray*}
Thus we have obtained the desired formula for the Hilbert-Kunz function of the ring $R$.
\end{proof}

\begin{remark}
Corollary \ref{hkfunctionfermatperiodicity} matches for $d=3$ with the result \cite[Theorem 4]{buchweitzchenhilbertkunz} of Buchweitz and Chen, which says that the Hilbert-Kunz function of the homogeneous coordinate ring of a plane elliptic curve defined over a field $K$ of odd characteristic $p$ equals $\frac{9}{4}p^{2e}-\frac{5}{4}$.
\end{remark}

\section{Examples of a (0,1)-Frobenius periodicity on Fermat curves}

In this section, we show how to get via Theorem \ref{cotangentfermatrepetition} non-trivial examples of $(0,1)$-Frobenius periodicities, i.e., we give explicit examples of vector bundles $\shE$ on certain Fermat curves such that $\shE \cong F^*(\shE)$.

\begin{example}
\label{exampleexplicitperiodicity}
Let $d \geq 2$ and let $K$ be a field of characteristic $p \equiv -1 \!\! \mod 2d$. The ring homomorphism
$$K[X,Y,Z]/(X^d+Y^d-Z^d) \lto K[U,V,W]/(U^{2d}+V^{2d}-W^{2d})$$
which sends $X \mapsto U^2$, $Y \mapsto V^2$ and $Z \mapsto W^2$ induces a finite cover
$f: C^{2d} \ra C^d$, where $C^i$ denotes the Fermat curve of degree $i$. Since $f^*(\O_{C^d}(1)) \cong \O_{C^{2d}}(2)$, we see that $\deg(f)=4$. The group ${\mathbb Z}/(2) \times {\mathbb Z}/(2)$ acts on $C^{2d}$ by sending $(u,v,w)$ either to $(u,v,w), (-u,v,w), (u,-v,w)$ or $(u,v,-w)$, and $C^d$ is the quotient of this action. Moreover, $f$ is a finite separable morphism and therefore $f$ preserves semistability. Theorem \ref{cotangentfermatrepetition} gives, via pull-back under $f$,
the isomorphic vector bundles
\begin{eqnarray*}
\Syz_{C^{2d}}(U^{2p},V^{2p},W^{2p}) &\cong& f^*(\Syz_{C^{d}}(X^p,Y^p,Z^p))\\
&\cong& f^*(\Syz_{C^{d}}(X,Y,Z)(-\frac{3(p-1)}{2}))\\
&\cong& f^*(\Syz_{C^{d}}(X,Y,Z)) \otimes f^*(\O_{C^{d}}(-\frac{3(p-1)}{2}))\\
&\cong& \Syz_{C^{2d}}(U^{2},V^{2},W^{2})(-3(p-1))
\end{eqnarray*}
on the Fermat curve $C^{2d}$. In particular, we have the periodicity
$$\Syz_{C^{2d}}(U^{2},V^{2},W^{2})(3) \cong F^*(\Syz_{C^{2d}}(U^{2},V^{2},W^{2})(3)) \, .$$
Note that this bundle has degree $0$ and is not trivial since there are no non-trivial global sections below the degree of the curve.
For $d \geq 4$ this bundle is stable, see Remark \ref{Verschiebung} below.
\end{example}

\begin{remark}
\label{etaletrivialization}
By the classical result \cite[Satz 1.4]{langestuhler} of H. Lange and U. Stuhler the periodicity
$\Syz_{C^{2d}}(U^{2},V^{2},W^{2})(3) \cong F^*(\Syz_{C^{2d}}(U^{2},V^{2},W^{2})(3))$ in Example \ref{exampleexplicitperiodicity} (in certain prime characteristics) implies the existence of an \'{e}tale cover $$g:D \lto C^{2d}$$ such that
$$g^*(\Syz_{C^{2d}}(U^{2},V^{2},W^{2})(3)) \cong \O_D^2.$$
Moreover, by \cite[Proposition 1.2]{langestuhler} the bundle $\Syz_{C^{2d}}(U^{2},V^{2},W^{2})(3)$ comes from a (continuous) representation $$\rho: \pi_1^{\et}(C^{2d}) \lto GL_2(K)$$ of the algebraic fundamental group $\pi_1^{\et}(C^{2d})$.
It would be interesting to see how the \'{e}tale trivialization $g$ and the representation $\rho$ look explicitly in this example.
For $2d=p+1$, $p$ an odd prime number, the trivialization was computed in \cite{staeblerexplicit}.
\end{remark}

\begin{remark}
In this remark we show that $\shE:=\Syz(U^2,V^2,W^2)(3)$ is not \'{e}tale trivializable in characteristic $0$.
We consider this bundle on the smooth projective relative curve
$$\shC^{2d}:=\Proj(\ZZ_{2d}[U,V,W]/(U^{2d}+V^{2d}-W^{2d})) \lto \Spec \ZZ_{2d}.$$ For a prime number $p \not| 2d$ the special fiber
$\shC^{2d}_p$ over $(p)$ is the (smooth) Fermat curve over the finite field $\FF_p$. The generic fiber $\shC^{2d}_0$ over $(0)$ is the Fermat curve
over $\QQ$. To prove that $\shE_0:=\shE|_{\shC^{2d}_0}$ is not \'{e}tale trivializable on $\shC^{2d}_0$ for $d \geq 4$ we use
\cite[Corollary 2]{brennermiyaoka}. By this result the bundle  $\shE^{2d}_0$ is semistable on  $\shC^{2d}_0$ in characteristic $0$,
but $\shE_p:=\shE|_{\shC^{2d}_p}$ is not strongly semistable in prime
characteristics $ p \equiv d - 1 \!\! \mod 2d$.
Note that by the well-known theorem of Dirichlet (cf. \cite[Chapitre VI, \S 4, Th\'{e}or\`{e}me and Corollaire]{serrearithmetic})
there are infinitely many such fibers. Now if $\shE_0 $ were \'{e}tale trivializable,
then we could extend the data which constitute the \'{e}tale cover and the trivialization over the generic point to an open subset of $\Spec \ZZ$.
But then for almost all prime numbers $p$ the bundle $\shE_p$ were  \'{e}tale trivializable and hence strongly semistable, since semistability descends
from a finite cover. Therefore, there is no \'{e}tale cover $g:D \ra \shC_0^{2d}$ such that $g^*(\shE_0) \cong \O_D^2$.

This observation is somehow related to the Grothendieck-Katz $p$-curvature conjecture \cite[(I quat)]{katzdifferential} which states the following: Let $R$ be a $\ZZ$-domain of finite type, $\ZZ \subseteq R$, and $\shX \ra \Spec R$ a smooth projective morphism of
relative dimension $d\geq 1$. If $\shE$ is a vector bundle on $\shX \ra \Spec R$ equipped with an integrable connection $\nabla$ such that $\nabla|_{\shE_\fop}$ has $p$-curvature $0$ on the special fiber $\shX_{\fop}$ for almost all closed points $\fop \in \Spec R$, then  there exists an \'{e}tale cover $g: Y \ra \shX_0$ of the generic fiber $\shX_0$ such that $(g^*(\shE_0),g^*(\nabla_0))$ is trivial. For a detailed account on integrable connections and $p$-curvature see \cite{katznilpotent} and \cite{katzdifferential}.
In our example of the relative Fermat curve $\shC^{2d}$, we have for infinitely many prime numbers $p \equiv -1 \!\! \mod 2d$ the Frobenius descent $F^*(\shE_p) \cong \shE_p$ on $\shC^{2d}_p$. By the so-called Cartier-correspondence
\cite[Theorem 5.1]{katznilpotent} this is equivalent to the existence of an integrable connection $\nabla_p$ on $\shE_p$ with
vanishing $p$-curvature. If one could establish a connection on $\shE$ (since $\shE$ is a vector bundle over a curve, this connection would be automatically integrable) which is compatible with the connections on the special fibers $\shC^{2d}_p$, $p \equiv -1 \mod 2d$, 
then our example would show that the Grothendieck-Katz conjecture does not hold if one only requires vanishing $p$-curvature for infinitely many closed points.
\end{remark}

\begin{remark}
\label{Verschiebung}
In this remark we assume that the base field is algebraically closed. We consider the \emph{Verschiebung}
$$V: \shM_{C^{2d}}(2,\O_{C^{2d}}) \dashrightarrow \shM_{C^{2d}}(2,\O_{C^{2d}}),~[\shE] \longmapsto [F^*(\shE)]$$
induced by the Frobenius morphism on $C^{2d}$. We recall that the Verschiebung is a rational map from the moduli space $\shM_{C^{2d}}(2,\O_{C^{2d}})$ parametrizing (up to $S$-equivalence) semi\-stable vector bundles on $C^{2d}$ of rank $2$ and trivial determinant to itself. The vector bundle $\shS:=\Syz_{\PP^2}(U^2,V^2,W^2)$ is stable on the projective plane $\PP^2=\Proj K[U,V,W]$ by \cite[Corollary 6.4]{brennerlookingstable}. Since the discriminant of this bundle equals $\Delta(\shS)=4c_2(\shS)-c_1(\shS)^2=12$, the restriction of this bundle to every smooth projective curve of degree $\geq 7$ remains stable by Langer's restriction theorem \cite[Theorem 2.19]{langersurvey}. In particular, $\shS|_{C^{2d}} \cong \Syz_{C^{2d}}(U^{2},V^{2},W^{2})(3)$ is stable on the Fermat curve $C^{2d}$ for $d \geq 4$. Hence, for $d \geq 4$ the bundle $\Syz_{C^{2d}}(U^{2},V^{2},W^{2})(3)$ defines a closed point of $\shM_{C^{2d}}(2,\O_C)$ which is fixed under the Verschiebung $V$.
\end{remark}

\begin{remark}
We may pull-back the vector bundle $\shE = \Syz(U^2,V^2,W^2)(3)$ along the cone mapping \[p:T= \Spec K[U,V,W]/(U^{2d}+V^{2d} - W^{2d}) \setminus \{ \fom \} \rightarrow C^{2d}\] to obtain the bundle $\shG=p^*(\shE)$ on the punctured spectrum with the property $F^*(\shG) \cong \shG$. This can however not be extended to get a Frobenius periodicity on the module level, since $F^* \Gamma(T, \shG) \neq \Gamma (T, F^* \shG)$. A Frobenius periodicity for a coherent $R$-module $M$, where $R$ is a local noetherian domain, implies that $M$ is free. This observation follows by looking at Fitting ideals of a free resolution (we thank Manuel Blickle and Neil Epstein for this remark).
\end{remark}

\bibliographystyle{amsplain}

\bibliography{bibliothek}



\end{document}